\documentclass[11pt,reqno]{amsart}
\usepackage{amssymb}
\usepackage{amsmath}
\usepackage{amsthm}
\usepackage{mathrsfs}
\usepackage{graphicx}
\usepackage{caption}
\usepackage{bm}
\usepackage[colorlinks=true,urlcolor=blue,
citecolor=red,linkcolor=blue,linktocpage,pdfpagelabels,
bookmarksnumbered,bookmarksopen]{hyperref}
\usepackage[titletoc,title]{appendix}

\usepackage[normalem]{ulem}

\usepackage{mathtools}	

\numberwithin{equation}{section} 
\newcounter{cont}[section] 

\newtheorem{thm}[cont]{Theorem}
\newtheorem{prop}[cont]{Proposition}
\newtheorem{lem}[cont]{Lemma}
\newtheorem{cor}[cont]{Corollary}
\theoremstyle{definition}

 \theoremstyle{remark}
 \newtheorem{rem}[cont]{Remark}

\newcommand{\R}{\mathbb{R}}
\newcommand{\e}{\varepsilon}

\DeclareMathOperator{\sgn}{sgn}	

\begin{document}
\baselineskip=16pt

\title[Time-fractional Allen--Cahn equations]{Slow motion for the time-fractional Allen--Cahn equation with linear and nonlinear diffusion}

\author[D. Cusseddu]{Davide Cusseddu}
\address[Davide Cusseddu]{Dipartimento di Scienze Matematiche ``G. L. Lagrange'' Politecnico di Torino, Corso Duca degli Abruzzi 24, 10129 Torino (Italy)}
\email{davide.cusseddu@gmail.com}

\author[R. Folino]{Raffaele Folino}
\address[Raffaele Folino]{Departamento de Matem\'aticas y Mec\'anica, Instituto de Investigaciones en Matem\'aticas Aplicadas y en Sistemas, Universidad Nacional Aut\'onoma de M\'exico, Circuito Escolar s/n, C.P. 04510 Cd. de M\'exico (M\'exico)}
\email{folino@aries.iimas.unam.mx}
 
\author[L. F. L\'{o}pez R\'{\i}os]{Luis F. L\'{o}pez R\'{\i}os}
\address[L. F. L\'{o}pez R\'{\i}os]{Departamento de Matem\'aticas y Mec\'anica\\Instituto de Investigaciones en Matem\'aticas Aplicadas y 
en Sistemas\\Universidad Nacional Aut\'{o}noma de M\'{e}xico\\Circuito Escolar s/n, Ciudad Universitaria, C.P. 04510\\Cd. de M\'{e}xico (Mexico)}
\email{luis.lopez@mym.iimas.unam.mx}


\maketitle

\begin{abstract}
We study the slow motion of phase transition layers for a time-fractional Allen--Cahn equation with a Caputo time derivative of order $\alpha\in(0,1)$. 
We show that the fractional memory preserves the distinction between exponentially and algebraically slow dynamics 
associated with non-degenerate and degenerate potentials, respectively. 
In the non-degenerate case, layered solutions persist on an exponentially long time scale of order $\exp(A/(\alpha\e))$, $A>0$,
while in the degenerate case the corresponding algebraic time scale is modified by the factor $1/\alpha$. 
Finally, we extend the analysis to nonlinear diffusion operators, considering both the $p$-Laplacian and a Perona--Malik type diffusion, 
showing that the same fractional slow-motion mechanism persists in these more general settings.
\end{abstract}

\section{Introduction}\label{sec:intro}
The interplay between local interactions and spatial diffusion gives rise to a rich variety of spatio-temporal dynamics, making reaction-diffusion equations a fundamental class of nonlinear evolution equations. A classical example is the Allen--Cahn equation \cite{Allen-Cahn}, which combines a diffusion operator with a nonlinear reaction term and has been extensively studied from both analytical and variational perspectives. In its classical formulation, the evolution is local in time and governed by the first-order time derivative. A different temporal structure arises when this derivative is replaced by a fractional one, which is nonlocal in time, as its value at a given time depends on the behavior of the solution over the entire preceding time interval. 
In this paper, we consider a time-fractional version of the classical Allen--Cahn equation, and investigate how time-fractional nonlocality affects the slow dynamics of phase transition layers with respect to the classical cases.

\subsection{Time-fractional Allen--Cahn equation}
We first investigate the persistence of phase transition layer solutions to the time-fractional Allen--Cahn equation
\begin{equation}\label{eq:C-model}
	\prescript{C}{}\partial_t^\alpha u=\e^2u_{xx}-F'(u),
\end{equation}
where $u=u(x,t) : [a,b]\times(0,+\infty)\rightarrow \R$, complemented with homogeneous Neumann boundary conditions
\begin{equation}\label{eq:Neu}
	u_x(a,t)=u_x(b,t)=0, \qquad \qquad t>0,
\end{equation}
and initial datum
\begin{equation}\label{eq:initial}
	u(x,0)=u_0(x), \qquad \qquad x\in[a,b].
\end{equation}
In \eqref{eq:C-model} ${}^C\partial_t^\alpha u$ denotes the Caputo fractional derivative of order $\alpha\in(0,1)$ of the function $u$ with respect to the time variable $t$, 
$\e>0$ is a small parameter, while $F:\R\to\R$ is a double well potential with wells of equal depth.
More precisely, the Caputo fractional time derivative of order $\alpha\in(0,1)$ is here defined by
\begin{equation*}\label{eq:Caputo}
	\prescript{C}{}\partial_t^\alpha u=\prescript{C}{0}D_t^\alpha u(x,t):=\frac{1}{\Gamma(1-\alpha)}\int_0^t \frac{u_t(x,s)}{(t-s)^\alpha}\,ds, \qquad t>0,
\end{equation*}
or, in a more compact form, by
\begin{equation*}
	\prescript{C}{}\partial_t^\alpha u=\prescript{}{0}I_t^{1-\alpha} u_t, \qquad t>0,
\end{equation*}
where 
\begin{equation*}\label{eq:RL-fi}
	\prescript{}{0}I_t^{\alpha} v:=\frac{1}{\Gamma(\alpha)}\int_0^t \frac{v(s)}{(t-s)^{1-\alpha}}\,ds,\qquad t>0,
\end{equation*}
is the Riemann-Liouville fractional integral of order $\alpha$ of a function $v$ and $\Gamma$ is the standard gamma function, see e.g. \cite{Diethelm,Kilbas}. 
On the right-hand side, for the reaction term we require that the potential $F\in C^1(\R)$ satisfies
\begin{equation}\label{ipoF1}
	F(\pm1)=F'(\pm1)=0, \qquad \quad F(u)>0 \quad \forall\,u\neq\pm 1,
\end{equation}
and that there exist constants $0<\lambda_1\leq\lambda_2$, $\eta>0$ and $\theta>1$ such that:
\begin{equation}\label{ipoF2}
	\lambda_1 |1\pm u|^{\theta-2} \leq \frac{F'(u)}{u\pm1}\leq\lambda_2 |1\pm u|^{\theta-2}, \qquad\qquad \mbox{ for } \qquad |u\pm1|<\eta.
\end{equation}
Therefore, \eqref{ipoF1} ensures that $F$ is a double well potential with wells of equal depth in $u=\pm1$ and \eqref{ipoF2}
describes the behavior of $F$ close to the minimal points.
In particular, notice that by integrating \eqref{ipoF2} and using \eqref{ipoF1}, we obtain
\begin{equation}\label{eq:ass-F3}
		\frac{\lambda_1}{\theta}|1\pm u|^{\theta}\leq F(u)\leq \frac{\lambda_2}{\theta}|1\pm u|^{\theta}, 
		\quad \qquad |u\pm1|<\eta.
\end{equation}
The simplest example of potential $F$ satisfying \eqref{ipoF1}-\eqref{ipoF2} is given by
\begin{equation}\label{F:ex}
	F(u)=\frac{1}{2\theta}|1-u^2|^\theta, \quad \theta>1.
\end{equation}
The corresponding profiles for different values of $\theta>1$ are illustrated in Figure \ref{fig:F}.

\begin{figure}[t]
	\begin{center}
		\includegraphics[width=8cm,height=6cm]{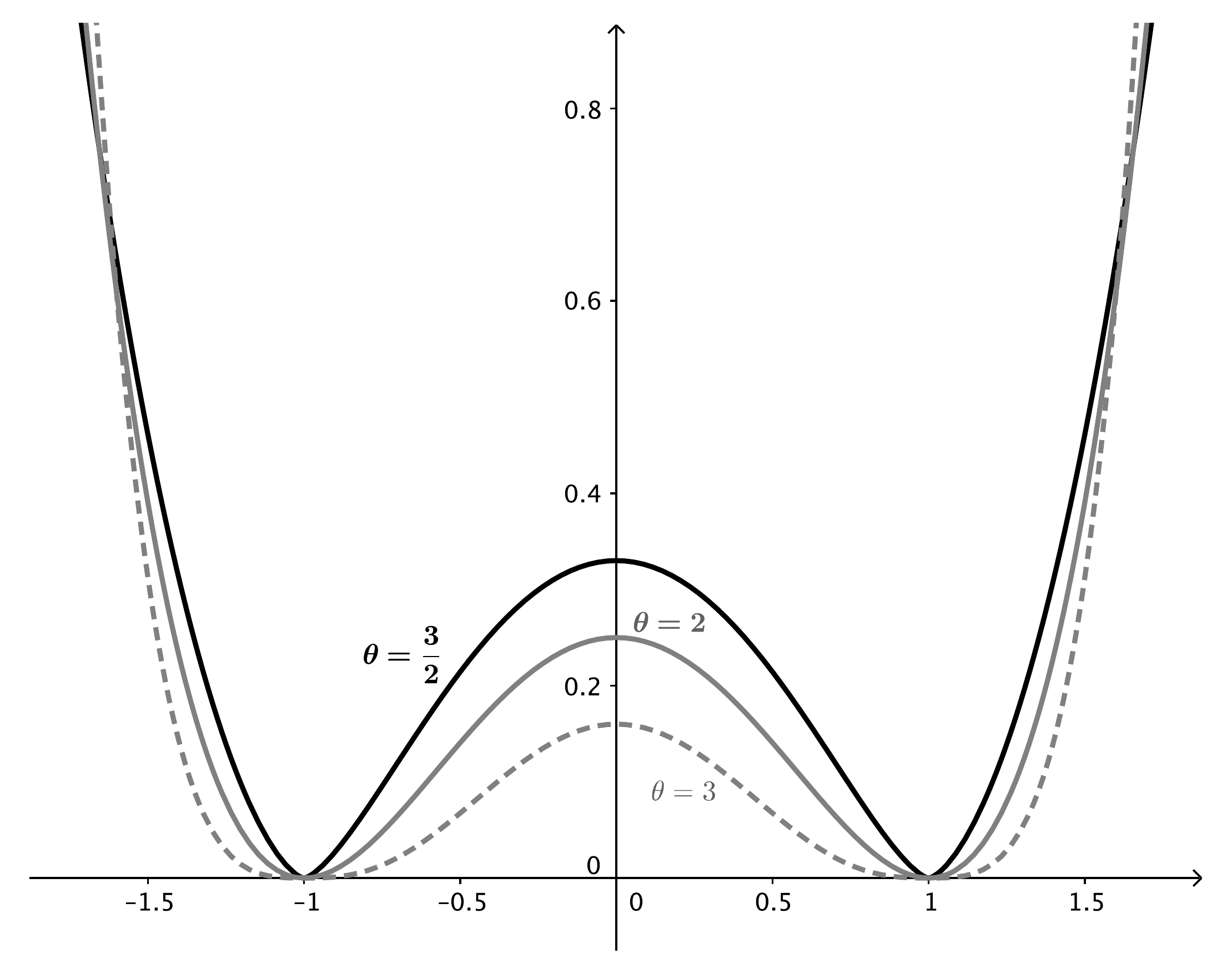}
		\hspace{3mm}
		\caption{\small{The potential $F$ defined in \eqref{F:ex} for different values of $\theta$.}}
		\label{fig:F}
	\end{center}
\end{figure}

It is worth mentioning that when $\theta=2$ in \eqref{F:ex}, we obtain the classical double well potential $F(u)=\frac{1}{4}(1-u^2)^2$, 
which, in particular, satisfies $F''(\pm1)>0$; 
on the other hand, if $\theta\in(1,2)$ the second derivative of the potential \eqref{F:ex} blows up as $u\to\pm1$, 
while for $\theta>2$ we have the \emph{degenerate} case $F''(\pm1)=0$. 
Finally, notice that \eqref{ipoF1} implies that the reaction term $f=-F'$ satisfies $\displaystyle\int_{-1}^1 f(s)\,ds=0$, 
being the reason why we call $f$ a balanced bistable reaction term.

The interplay between a balanced reaction term satisfying $F''(\pm1)>0$ and the linear diffusion has been extensively studied 
since the seminal works \cite{Bron-Kohn,Carr-Pego,Fusco-Hale},
where the authors consider the classical one dimensional Allen--Cahn equation \cite{Allen-Cahn}, 
formally obtained from \eqref{eq:C-model} by setting $\alpha=1$:
\begin{equation}\label{eq:A-C}
	u_t=\e^2u_{xx}-F'(u).
\end{equation}
For small $\e>0$, solutions with transition layer initial data exhibit \emph{metastable dynamics}, namely the interfaces move with exponentially small velocity until their annihilation, 
after which the remaining layers undergo the same slow dynamics. 
Eventually, the solution converges to one of the stable equilibria $u=\pm1$.
A rigorous analysis of this phenomenon was initiated in \cite{Carr-Pego}, where the annihilation time of the closest layers was shown to be of order $\exp(Al/\e)$, 
with $A=\min{F''(\pm1)}>0$ and $l$ the distance between the interfaces. 
The condition $F''(\pm1)>0$ is essential for the metastable dynamics, whereas in the degenerate case $F''(\pm1)=0$
the exponentially slow motion is replaced by an algebraic one \cite{Bet-Sme}.

For potentials of the form \eqref{F:ex} with $\theta\in(1,2)$, the behaviour differs significantly, 
since stationary solutions of \eqref{eq:A-C} with boundary conditions \eqref{eq:Neu} may exhibit an arbitrary number of transition layers located inside $(a,b)$ \cite{Dra-Rob,FPS-DCDS}. More generally, \cite{FPS-DCDS} shows that for the $p$-Laplacian extension, obtained by replacing $u_{xx}$ with $(|u_x|^{p-2}u_x)_x$, the dynamics is determined by the relation between $p$ and $\theta$: layered stationary states exist for $\theta<p$, metastability occurs for $\theta=p$, and algebraic slow motion takes place for $\theta>p$.

The main novelty of our work consists in extending the aforementioned results
to the setting of fractional time derivatives and in emphasizing the role played by the fractional order $\alpha\in(0,1)$.
Since the stationary solutions of \eqref{eq:C-model} and \eqref{eq:A-C} coincide, 
we restrict our analysis to the case where $F$ satisfies \eqref{ipoF1}-\eqref{ipoF2} with $\theta\geq2$.
Before presenting our results, we briefly recall some recent works concerning the time-fractional Allen--Cahn equation.

In applications, among the various definitions of fractional derivatives, the Caputo derivative is perhaps the most widely used. It is generally preferred to the Riemann--Liouville derivative, as it can be viewed as a regularisation of the latter. In particular, in the vast majority of the literature, the time-fractional Allen--Cahn equation is treated with the Caputo derivative.
	
Liu et al. \cite{Liu} suggest that time-fractional Allen--Cahn equation may be used to model diffusion in highly porous media, 
where the introduction of the Caputo derivative is advantageous for modeling sub-diffusive processes. 
It is worth remarking, however, that from a physical perspective, a clear derivation of the fractional model from the underlying physical processes remains necessary. 
Indeed, replacing the classical time derivative, i.e., the rate of variation, with a fractional derivative does not, by itself, necessarily translate into a new or more general physical model. Related concerns are discussed in \cite{Cusseddu} and, in the context of the time-fractional Fisher--KPP equation, in \cite{AngstmannHenry}.
Nevertheless, the time-fractional Allen--Cahn equation has already been studied by many authors, particularly from the numerical point of view. 
On the analytical side, it was shown by Tang et al. \cite{TYZ} that an energy dissipation law exists also in the time-fractional case, but of integral type, see also \cite{Quan23, Fritz2026}.
Du et al. \cite{DYZ2020} later complemented this result, by proving well-posedness of the time-fractional Allen--Cahn problem with homogeneous Dirichlet boundary conditions and derived an additional fractional energy dissipation law. 
The time-fractional Allen--Cahn equation was also studied by Dipierro et al. \cite{DNV24}, 
who formally showed that the level sets of solutions evolve according to a power of the mean curvature depending on the fractional order.

However, to the best of our knowledge, there has been no rigorous analysis of the slow dynamics 
of phase transition layer solutions to reaction-diffusion models with a time derivative of Caputo type.
In this direction, we stress once again that the stationary solutions are not affected by the introduction of the fractional time derivative;
indeed, since the Caputo derivative vanishes for time-independent functions, the stationary solutions of \eqref{eq:C-model} coincide with those of \eqref{eq:A-C}. 
Thus, the effect of the fractional time derivative does not change the set of stationary states, but rather affects the dynamics of evolving solutions.

For the non-degenerate case $\theta=2$ in \eqref{ipoF2}, Theorem \ref{thm:main} establishes the persistence of the metastable behavior known for 
the classical Allen--Cahn equation \eqref{eq:A-C} in the fractional framework \eqref{eq:C-model}. 
More precisely, solutions starting from a suitable layered initial datum remain close to the corresponding initial profile for an exponentially long time,
that is, a time of order $\exp\left(\frac{A}{\alpha\e}\right)$, where $A$ is a positive constant. 
A remarkable feature of this result is that the fractional order $\alpha$ appears in the denominator of the exponential time scale. 
Hence, the memory effect induced by the fractional derivative further slows down a dynamics which is already exponentially slow in the classical case.

For degenerate potentials, namely $\theta>2$ in \eqref{ipoF2}, the behavior is substantially different. 
Theorem \ref{thm:main2} shows that the exponentially slow motion is replaced by an algebraically slow dynamics, 
and that layered solutions persist for an algebraically long time, with time scale $\e^{-k_j/\alpha}$, for any $j\in\mathbb N$, where $k_j$ is defined in \eqref{eq:exp_alg}.  
Also in this case, the fractional order $\alpha$ enters the time scale through the factor $1/\alpha$, 
showing that the fractional memory further delays the evolution of the phase transition layers.

Finally, Theorem \ref{thm:interface} gives a direct description of the slow motion of the transition layers, 
showing that their displacement remains small on the same exponentially and algebraically long time scales described above.

These results show that the fractional time derivative preserves the qualitative distinction between non-degenerate and degenerate potentials, 
while further slowing down the corresponding dynamics. 
We next extend these results to more general nonlinear diffusion operators.
As we will see, in the nonlinear diffusion setting, a further source of degeneracy may arise from the diffusion operator itself. 

\subsection{Nonlinear diffusion} 
In the last part of this work, we consider two nonlinear diffusion models, obtained by replacing the linear diffusion term $\e^2u_{xx}$ either with 
the $p$-Laplacian operator $\e^p(|u_x|^{p-2}u_x)_x$, $p>1$, or with a nonlinear diffusion of Perona--Malik type, namely $Q(\e^2u_x)_x$, introduced in \cite{Per-Mal}. 
The two models read as
\begin{equation}\label{eq:p-lap}
	\prescript{C}{}\partial_t^\alpha u=\e^p(|u_x|^{p-2}u_x)_x-F'(u), \qquad p>1,
\end{equation}
and
\begin{equation}\label{eq:Q-dif}
	\prescript{C}{}\partial_t^\alpha u=Q(\e^2u_x)_x-F'(u).
\end{equation}
The $p$-Laplacian model provides a natural extension of the classical Allen--Cahn equation and allows the effect of the nonlinear (possibly degenerate) diffusion 
to be separated from that of the fractional time derivative. 
In the classical case, the slow dynamics of layered solutions depends on the relation between the exponent $p$ of the diffusion 
and the exponent $\theta$ describing the behavior of the potential $F$ near its minima. 
In particular, the critical case $\theta=p$ corresponds to exponentially slow motion, whereas for $\theta>p$ the slow motion is algebraic \cite{FPS-DCDS}.
Notice that the model \eqref{eq:C-model} is recovered for $p=2$ in \eqref{eq:p-lap}.

The Perona--Malik model \eqref{eq:Q-dif} presents a different type of nonlinear diffusion. 
We assume that the function $Q\in C^1(\R)$ satisfies
\begin{equation}\label{eq:Q-ass1}
	\lim_{s\to\pm\infty}Q(s)=Q(0)= 0, \qquad \qquad Q(-s)=-Q(s),
\end{equation}
for all $s\in\R$ and that there exists $\kappa>0$ such that
\begin{equation}\label{eq:Q-ass2}
	Q'(s)>0, \quad \forall \, s \in (-\kappa,\kappa) \qquad \mbox{ and } \qquad Q'(s)<0, \quad \mbox{ if } |s|>\kappa.
\end{equation}
The prototype examples are
\begin{equation}\label{ex:fluxfunction}
	Q(s):=\frac{s}{1+s^2} \qquad \mbox{ and } \qquad Q(s):=s e^{-s^2},
\end{equation}
which satisfy assumptions \eqref{eq:Q-ass2} with $\kappa = 1$ and $\kappa=1/\sqrt{2}$, respectively (see Figure \ref{fig:Q}).
\begin{figure}[ht]
	\begin{center}
		\includegraphics[width=8cm,height=6cm]{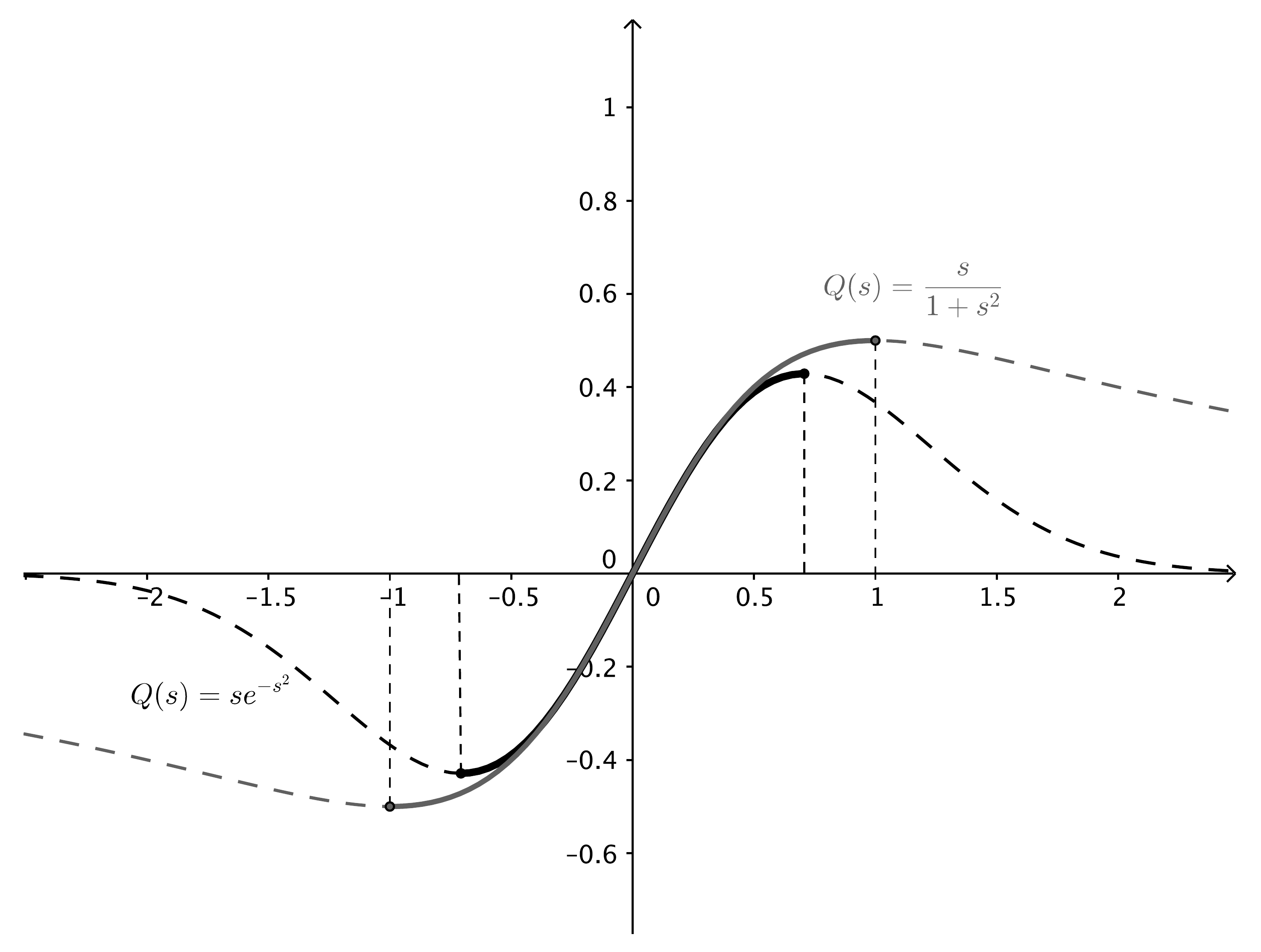}
		\hspace{3mm}
		\caption{\small{The two prototype examples defined in \eqref{ex:fluxfunction}.}}
		\label{fig:Q}
	\end{center}
\end{figure}

The $p$-Laplacian also allows for degenerate diffusion when $p>2$, since the diffusion coefficient $|u_x|^{p-2}$ vanishes at $u_x=0$. 
The Perona--Malik model exhibits a different type of nonlinearity, characterized by a sign-changing diffusion coefficient. 
Indeed, writing
$$
	Q(\e^2u_x)_x=\e^2Q'(\e^2u_x)u_{xx},
$$
the coefficient $Q'(\e^2u_x)$ is positive in the parabolic region and becomes negative when $|\e^2u_x|$ crosses a threshold. 
This sign-changing structure is the characteristic feature of the Perona--Malik model, originally introduced in \cite{Per-Mal}. 
The existence and slow dynamics of layered stationary solutions for this type of diffusion in the classical setting have been studied in \cite{DLFS}.

Before presenting the results for the models \eqref{eq:p-lap}-\eqref{eq:Q-dif}, 
we explain why we chose to begin the paper with the results for the model with linear diffusion \eqref{eq:C-model}, 
although this case is included in the $p$-Laplacian framework by taking $p=2$. 
The main reason is that it allows us to introduce the ideas underlying the extension to the fractional setting in the simplest and most classical case, 
without adding at the same time the technical difficulties related to the nonlinear diffusion. 
In particular, our aim is first to show how the results for the classical Allen--Cahn equation \eqref{eq:A-C} can be transferred to the case of a Caputo time derivative \eqref{eq:C-model}. 
Once this idea is understood in the simpler setting, the extension to more general diffusion operators follows by 
combining the same arguments with the results already available for the corresponding problems with the classical time derivative (formally obtained by taking $\alpha=1$), 
and contained in \cite{FPS-DCDS} for the $p$-Laplacian and in \cite{DLFS} for Perona--Malik type diffusions.
The results obtained in Section \ref{sec:nonlinear} show that the effect of the fractional time derivative is robust with respect to the choice of the diffusion operator. 
Indeed, the distinction between exponentially and algebraically slow motion induced by the potential is preserved, 
while the fractional memory produces the same additional slowing effect through the factor $1/\alpha$ in the corresponding time scales.

\subsection{Plan of the paper} 
The rest of the paper is organized as follows. 
In Section \ref{sec:pre}, we introduce the energy associated with the time-fractional Allen--Cahn equation \eqref{eq:C-model} 
and establish the energy dissipation estimate induced by the Caputo derivative. 
We also recall the variational results concerning the energy of layered functions. 
In Section \ref{sec:slow}, we use these estimates to study the slow motion of phase transition layers for  \eqref{eq:C-model}, 
considering both the non-degenerate potential case $\theta=2$ and the degenerate one $\theta>2$,
and then we describe the corresponding motion of the interfaces. 
Section \ref{sec:nonlinear} is devoted to nonlinear diffusion:
we first consider the time-fractional Allen--Cahn equation with $p$-Laplacian \eqref{eq:p-lap} and then the one with Perona--Malik type diffusion \eqref{eq:Q-dif}, 
showing that the same fractional slow-motion mechanism applies to these more general diffusion operators.

\section{Energy for the time-fractional Allen--Cahn equation}\label{sec:pre}
An energy approach used to investigate the slow motion of the solutions to the Allen--Cahn equation \eqref{eq:A-C} was firstly proposed by Bronsard and Kohn in \cite{Bron-Kohn} 
and it was later refined by Grant \cite{Grant} to study Cahn--Morral systems. 
More recently, it was applied in the analysis of different models of Allen--Cahn type, see e.g. \cite{DLFS,FPS-DCDS} and references therein.

Due to the presence of the time-fractional derivative $\prescript{C}{}\partial_t^\alpha u$ instead of $u_t$, the slow motion analysis of the solutions to \eqref{eq:C-model}-\eqref{eq:Neu}-\eqref{eq:initial} needs to be adapted, by incorporating results from fractional calculus. 

\subsection{Energy dissipation}
The energy associated to \eqref{eq:C-model}-\eqref{eq:Neu}-\eqref{eq:initial} is defined by
\begin{equation}\label{eq:energy}
	E_\e[u]:=\int_a^b \left[\frac{\e u_x^2}{2}+\frac{F(u)}\e\right]\, dx.
\end{equation}
The main result of this section presents a crucial inequality for \eqref{eq:energy}, 
which relates the variation of the energy $E_\e[u]$ evaluated along solutions of \eqref{eq:C-model}-\eqref{eq:Neu}-\eqref{eq:initial}
to the $L^2$-norm of the Caputo derivative of order $(1+\alpha)/2$.
To establish the energy dissipation we use the $L^2$-theory of the Fourier transform. For $u \in L^2(\R) \cap L^1(\R)$ this operator is defined by
\begin{equation*}
	\hat u (\omega) = \int_{-\infty}^\infty u(t) e^{-i \omega t} \, dt,
\end{equation*}
where $i=\sqrt{-1}$. For a general $u \in L^2(\R)$, the Fourier transform is defined by density, through the geometric structure of $L^2(\R)$. This is a complex Hilbert space with the inner product $(u,v) := \int_{-\infty}^\infty u \overline{v}$, with the bar denoting the complex conjugate. It is well known that the Fourier transform is a \emph{conformal transformation} of the $L^2(\R)$ space, namely
\begin{equation}\label{eq:conformal}
	(\hat u, \hat v) = (u,v), \quad \text{for all } u, v \in L^2(\R).
\end{equation}
Regarding fractional integrals, for $v \in L^2(\R) \cap L^1(\R)$ it is defined
\begin{equation*}
	\prescript{}{-\infty}{I}_t^\beta v (t)=\frac{1}{\Gamma(\beta)}\int_{-\infty}^t \frac{v(s)}{(t-s)^{1-\beta}}\,ds,\qquad \prescript{}{t}{I}_\infty^\beta v (t)=\frac{1}{\Gamma(\beta)}\int_t^\infty \frac{v(s)}{(s-t)^{1-\beta}}\,ds.
\end{equation*}
For $\beta \in (0,1/2)$, their Fourier transforms are
\begin{equation} \label{eq:frac_integrals}
    	(\prescript{}{-\infty}{I}_t^\beta v)^\wedge(\omega) = \frac{\hat v}{(i \omega)^\beta}, \qquad\qquad
	(\prescript{}{t}{I}_\infty^\beta v)^\wedge(\omega) = \frac{\hat v}{(-i \omega)^\beta},
\end{equation}
generalizing their corresponding classical formulae for repeated integration. 
The restrictions on $v$ and $\beta$ keep all the previous expressions in the $L^2(\R)$ space. Recall that
\begin{equation} \label{eq:complex_power}
	(\pm i \omega)^\beta = |\omega|^\beta e^{\pm \sgn(\omega) \beta \frac{\pi}{2}i},
\end{equation}
where $\sgn$ is the sign function. For more details, we refer the reader to \cite[Section~7.1]{SaKiMa93}.

\begin{prop}\label{lem:energy}
Let $u\in C([0,T],H^2(a,b))\cap C^1([0,T],H^1(a,b))$ be a solution of \eqref{eq:C-model}-\eqref{eq:Neu}.
Let $E_\e$ be the functional defined in \eqref{eq:energy}. Then,
\begin{equation}\label{eq:energyestimate}
	E_\e[u](0)-E_\e[u](T)\geq \e^{-1}\sin\left(\frac{\alpha\pi}{2}\right)\int_0^T\!\!\int_a^b\left[\prescript{C}{}\partial_t^{(1+\alpha)/2}u(x,t)\right]^2\,dxdt\geq0,
\end{equation}
for any $\alpha\in(0,1)$.
\end{prop}
\begin{proof}
By multiplying
\eqref{eq:C-model}
 by $u_t$ and integrating over the spatial domain $(a,b)$, we obtain 
 $$
\int_{a}^{b} 
\prescript{C}{}\partial_t^{\alpha}u(x,t) u_t(x,t) \, dx = \int_{a}^{b} \left[\e^2u_{xx}(x,t)-F'(u(x,t))\right]u_t(x,t)\,dx.
$$
By using integration by parts on the right-hand side and the boundary Neumann conditions \eqref{eq:Neu}, we deduce
$$
\int_{a}^{b} \left[\e^2u_{xx}(x,t)-F'(u(x,t))\right]u_t(x,t)\,dx=-\e\frac{d}{dt}E_\e[u](t).
$$
Hence,
$$
\e^{-1} \int_{a}^{b} 
\prescript{C}{}\partial_t^{\alpha}u(x,t) u_t(x,t) \, dx = - \frac{d}{dt}E_\e [u](t),
$$
and integrating with respect to time, we infer
$$
\e^{-1}  \int_{0}^{T}\!\int_{a}^{b}
\prescript{C}{}\partial_t^{\alpha}u(x,t) u_t(x,t) \, dx \, dt = E_\e [u](0) - E_\e [u](T).
$$
We claim that 
\begin{equation}\label{eq:derivatives_lucho}
\int_{0}^{T}\!\int_{a}^{b} 
\prescript{C}{}\partial_t^{\alpha}u(x,t) u_t(x,t) \, dxdt \geq 
\sin\left(\frac{\alpha\pi}{2}\right)\int_0^T\!\int_a^b\left[\prescript{C}{}\partial_t^{(1+\alpha)/2}u(x,t)\right]^2\,dxdt.
\end{equation}
Indeed, notice that, from the definition of Caputo derivative, it follows that
\begin{align*}
 	\int_{0}^{T}\!\int_{a}^{b}\prescript{C}{}\partial_t^{\alpha}u(x,t) u_t(x,t) \, dxdt &= \int_{0}^{T}\!\int_{a}^{b}\prescript{}{0}I_t^{1-\alpha} u_t(x,t) u_t(x,t) \, dxdt\\
	&=\int_{a}^{b}\left[\int_0^T \prescript{}{0}I_t^{(1-\alpha)/2} \left(\prescript{}{0}I_t^{(1-\alpha)/2}u_t(x,t)\right)u_t(x,t)\,dt\right]\,dx\\
	&=\int_{a}^{b}\left[\int_0^T\prescript{}{0}I_t^{(1-\alpha)/2} u_t(x,t)\,\prescript{}{t}I_T^{(1-\alpha)/2} u_t(x,t) \,dt\right]\,dx,
\end{align*}
where the second equality comes from the semigroup property of fractional integrals and the last one from the fractional integration by parts, 
see respectively Property 2.6 and 2.10 in \cite{Kilbas}. 
To compute the integral in time, we follow \cite[Theorem 3.6]{BC-LR} and use the Fourier transform (see also \cite[Corollary 2.1]{TYZ}). 
To this end, let us fix $x \in (a,b)$ such that $u_t(x,\cdot) \in L^2(0,T)$ and consider $v(x,\cdot) \in L^2(\R)$ as its extension by \emph{zero} outside of $(0,T)$. 
In this setting the integral is transformed into
\begin{align*}
	\int_\R \prescript{}{-\infty}{I}_t^{(1-\alpha)/2} v(x,t)\,\prescript{}{t}{I}_\infty^{(1-\alpha)/2} v(x,t) \,dt & = \left( \prescript{}{-\infty}{I}_t^{(1-\alpha)/2} v(x,t)\, , \, \prescript{}{t}{I}_\infty^{(1-\alpha)/2} v(x,t) \right) \\
	& = \left( \left( \prescript{}{-\infty}{I}_t^{(1-\alpha)/2} v(x,t) \right)^\wedge \, , \, \left( \prescript{}{t}{I}_\infty^{(1-\alpha)/2} v(x,t) \right)^\wedge \right) \\
	&= \int_\R \frac{|\hat v(x,\omega)|^2}{(i\omega)^{(1-\alpha)/2}\overline{(-i\omega)^{(1-\alpha)/2}}}\, d\omega,
\end{align*}
by \eqref{eq:conformal} and \eqref{eq:frac_integrals}. 
Now observe that \eqref{eq:complex_power} implies 
$$(i\omega)^{(1-\alpha)/2}\overline{(-i\omega)^{(1-\alpha)/2}} = |\omega|^{1-\alpha} e^{\sgn(\omega)\frac{1-\alpha}{2}\pi i},$$ 
and thereby  
\begin{align*}
	\int_\R \frac{|\hat v(x,\omega)|^2}{(i\omega)^{(1-\alpha)/2}\overline{(-i\omega)^{(1-\alpha)/2}}}\, d\omega &= 
	\int_\R \frac{|\hat v(x,\omega)|^2}{|\omega|^{1-\alpha}} e^{-\sgn(\omega)\frac{1-\alpha}{2}\pi i} \, d\omega \\
	&=  \int_{-\infty}^0 \frac{|\hat v(x,\omega)|^2}{|\omega|^{1-\alpha}} e^{\frac{1-\alpha}{2}\pi i} \, d\omega +  
	\int_0^\infty \frac{|\hat v(x,\omega)|^2}{|\omega|^{1-\alpha}} e^{-\frac{1-\alpha}{2}\pi i} \, d\omega \\
	&= \cos\left( \frac{1-\alpha}{2} \pi \right) \int_\R \frac{|\hat v(x,\omega)|^2}{|\omega|^{1-\alpha}} \, d\omega \\
	&=  \sin\left( \frac{\alpha \pi}{2} \right) \left\| \prescript{}{-\infty}{I}_t^{(1-\alpha)/2} v(x,\cdot) \right\|_{L^2(\R)} \\
	&\ge \sin\left(\frac{\alpha\pi}{2}\right) \left\| \prescript{C}{}{\partial}_t^{(1+\alpha)/2}u(x,\cdot)\right\|^2_{L^2(0,T)}.
\end{align*}
Integrating over $(a,b)$, we end up with \eqref{eq:derivatives_lucho} and the proof is complete.
\end{proof}

As we will see in Sections \ref{sec:slow}-\ref{sec:nonlinear}, inequality \eqref{eq:energyestimate} plays a crucial result in the proof of the slow motion phenomena.
Notice that since $\sin(\alpha\pi/2)>0$ for any $\alpha\in(0,1)$, then from \eqref{eq:energyestimate} we deduce that
\begin{equation}\label{remark:E(t)<E(0)}
	E_\e[u](t) \leq E_\e[u](0),
\end{equation}
for any $t>0$, see also \cite{TYZ}.
\subsection{Variational results}
Here, we recall some variational properties of the energy \eqref{eq:energy}. 
The following are well known results concerning \eqref{eq:energy} and equation \eqref{eq:C-model} plays no role in what follows.
First, it is well known that if $\bar{u}$ is a monotone function connecting the two stable states $+1$ and $-1$, then the energy \eqref{eq:energy} satisfies
\begin{equation}\label{eq:c_0}
	E_\e[\bar{u}]\geq\sqrt2\int_a^b|\bar{u}_x|\sqrt{F(\bar{u})}\,dx=\sqrt2\int_{-1}^{+1}\sqrt{F(s)}\,ds=:c_2,
\end{equation}
The positive constant $c_2$ defined in \eqref{eq:c_0} represents the minimum energy to have a {\it single} transition between $-1$ and $+1$ in the following sense:
let us fix, throughout this paper, a {\it piecewise constant function} $v$ 
and a number $N\in\mathbb{N}$ of transitions,
as follows:
\begin{equation}\label{vstruct}
	\begin{aligned}
	v:[a,b]\rightarrow\{-1,+1\}\  \hbox{with $N$ jumps located at $a<h_1<h_2<\cdots<h_N<b$ and}\ r>0\\
	\hbox{such that}\ (h_i-r,h_i+r)\cap(h_j-r,h_j+r)=\emptyset \  \hbox{for}\ i\neq j\ \hbox{and}\   a\leq h_1-r,\ h_N+r\leq b.
	\end{aligned}
\end{equation}
It has been proved that if $\{ u^\e\}_{\e >0}$ is a family of functions sufficiently close to $v$ in $L^1$, then 
$$E_\e[{u^\e}] \geq N c_\e-R_{\theta,\e},$$ 
where the reminder term $R_{\theta,\e}$ goes to zero as $\e \to 0^+$ with a speed rate depending on $\theta$.
If $\theta=2$, we have the following result.
\begin{prop}\label{prop:lower}
Assume that $F\in C^1(\R)$ satisfies \eqref{ipoF1}-\eqref{ipoF2} with $\theta=2$. 
Let $v$ be as in \eqref{vstruct} and $A\in(0,r\sqrt{2\lambda_1})$ with
$\lambda_1>0$ (independent on $\e$) as in \eqref{ipoF2}.
Then, there exist $\e_0,C,\delta>0$ (depending only on $F,v$ and $A$) such that if $u\in H^1(a,b)$ satisfies 
\begin{equation}\label{eq:u-v}
	\|u-v\|_{{}_{L^1}}\leq\delta,
\end{equation}
then for any $\e\in(0,\e_0)$,
\begin{equation}\label{eq:lower}
	E_\varepsilon[u]\geq Nc_2-C\exp(-A/\varepsilon),
\end{equation}
where $E_\e$ and $c_2$ are defined in \eqref{eq:energy} and \eqref{eq:c_0}, respectively.
\end{prop}

\begin{proof}
The lower bound \eqref{eq:lower} was firstly proved by Grant in \cite[Lemma 2.1]{Grant} with a different assumption on the function $u$.
Later, a similar result was proved in \cite[Proposition 3.2]{FPS-DCDS} for the energy
$$\int_a^b\left[\frac{\e^{p-1}|u_x|^p}{p}+\frac{(u^2-1)^p}{2p\e}\right]\,dx, \qquad \qquad \mbox{ for any } p>1,$$
and in \cite[Proposition 3.4]{DLFS} for the energy
$$\int_a^b\left[\frac{\tilde Q(\e^2u_x)}{\e^3}+\frac{F(u)}{\e}\right]\,dx,$$
for a generic function $\tilde Q$ satisfying appropriate assumptions and a potential $F$ satisfying \eqref{ipoF1}-\eqref{ipoF2} with $\theta=2$. 
Since the three latter results are stated in greater generality, their proofs can be readily adapted to the particular setting considered in this paper,
see also Propositions \ref{prop:lower-p}-\ref{prop:lower-Q} for more details.
\end{proof}

In the case $\theta>2$, it is possible to prove a {\it weaker} lower bound for the energy, which is stated in the following proposition.
\begin{prop}\label{prop:lower_deg}
Assume that $F\in C^1(\R)$ satisfies \eqref{ipoF1}-\eqref{ipoF2} with $\theta>2$. 
Let $v$ be as in \eqref{vstruct} and define the sequence
\begin{equation}\label{eq:exp_alg}
	k_j:=\sum_{m=1}^j\gamma^m,
	\qquad\qquad  \mbox{ where }  \quad \gamma:=\displaystyle\frac{1}{2}+\frac{1}{\theta}.
\end{equation}
Then, for any $j\in \mathbb N$ there exist constants $\delta_j>0$ and $C_j>0$ such that if $u\in H^1(a,b)$ satisfies
\begin{equation}\label{|w-v|_L^1<delta_l}
	\|u-v\|_{L^1}\leq\delta_j,
\end{equation}
and 
\begin{equation}\label{E_p(w)<Nc0+eps^l}
	E_\e[u]\leq Nc_2+C\varepsilon^{k_{j}},
\end{equation}
for some $C>0$, with $\varepsilon$ sufficiently small, and $c_2$ defined in \eqref{eq:c_0},then
\begin{equation}\label{E_p(w)>Nc0-eps^l}
	E_\e[u]\geq Nc_2-C_j\varepsilon^{k_{j+1}}.
\end{equation}
\end{prop}

\begin{proof}
Again, the result is proved in some more general cases (see \cite[Proposition 4.1]{FPS-DCDS} and \cite[Proposition 3.5]{DLFS})
and its proof can be adapted without much extra bookkeeping. 
\end{proof}

Notice that if $\theta>2$ then $\gamma\in(0,1)$, and, consequently the increasing sequence defined in \eqref{eq:exp_alg} is bounded and satisfies
\begin{equation}\label{eq:beta}
	\lim_{j\to+\infty}k_j= \sum_{m=1}^{+\infty} \gamma^m =\frac{1}{1-\gamma}-1=\frac{\theta+2}{\theta-2}:=\beta.
\end{equation}
Therefore, when $\theta>2$ we only have the lower bounds 
\begin{equation*}
	E_\e[u]\geq Nc_\e-C_j\varepsilon^{k_j}, \qquad \mbox{for any } j\in\mathbb N,\qquad \mbox{ with } \qquad \lim_{j\to+\infty}k_j=\beta.
\end{equation*}

The next corollary follows directly from the above Propositions \ref{prop:lower}-\ref{prop:lower_deg}. 
It states that, if the set of functions $\{ u_\e\}_{\e >0}$ makes $N$ transitions between the stable states $-1$ and $+1$ in an {\it energy efficient way}, 
meaning in the sense of \eqref{E_p(w)<Nc0+eps^l}, 
then, when $\e$ tends to 0, 
$E_\e[u^\e]$ converges to the minimum energy in the case of the classical Ginzburg-Landau functional (see \cite{Bron-Kohn}).

\begin{cor}\label{cor:TLS}
Let $v$ be as in \eqref{vstruct} and let $u^\e \in H^1(a,b)$ be such that 
\begin{equation}\label{TLS1}
	\lim_{\varepsilon\rightarrow 0^+} \| u^\e -v \|_{L^1}=0,
\end{equation}
and \eqref{E_p(w)<Nc0+eps^l} holds true. 
Then,
\begin{equation*}
	\lim_{\e \to 0^+} E_\e[u^\e] = N c_2.
\end{equation*}
\end{cor}

\begin{proof}
Using the results of Propositions \ref{prop:lower}-\ref{prop:lower_deg} and applying \eqref{E_p(w)<Nc0+eps^l} it is easy to see that
\begin{equation}\label{eq:proofTLS}
	Nc_2 -C R_{\theta, \e} \leq E_\e[u^\e] \leq Nc_2 + C\e^{k_j},
\end{equation}
where 
\begin{equation}\label{def:Rte}
R_{\theta, \e}= \left\{ \begin{aligned}
& \exp(-A/\varepsilon) \quad  &\mbox{if}  \quad &\theta=2 \\
& \e^{k_j}  \quad  &\mbox{if}  \quad &\theta>2,
\end{aligned}\right.
\end{equation}
and $ k_j $ is defined in \eqref{eq:exp_alg}. 
Hence, it is now sufficient to pass to the limit as $\e \to 0^+$ in \eqref{eq:proofTLS} to obtain the thesis.
\end{proof}

We conclude this section by recalling how to construct a family of functions satisfying \eqref{E_p(w)<Nc0+eps^l}-\eqref{TLS1}.
To do this, we will use the following standing waves solutions
\begin{equation}\label{eq:Fi}
	\e^2\Phi_\e''-F'(\Phi_\e)=0, \quad \qquad \lim_{x\to\pm\infty}\Phi_\e(x)=\pm1, \qquad \Phi_\e(0)=0.
\end{equation}
Fix $N$ transition points $a<h_1<h_2<\dots<h_n<b$, and denote the middle points by
\begin{equation*}
	m_1:=a, \qquad \quad m_j:=\frac{h_{j-1}+h_j}{2}, \quad j=2,\dots,N, \qquad \quad m_{N+1}:=b.
\end{equation*}
Define
\begin{equation}\label{eq:translayer}
	u^\e(x):=\Phi_\e\left((-1)^j(x-h_j)\right), \qquad \qquad x\in[m_j,m_{j+1}], \qquad \qquad j=1,\dots N,
\end{equation}
where $\Phi_\e$ is as in \eqref{eq:Fi}. 
Notice that $\Psi_\e(x):=\Phi_\e(-x)$ is still a standing wave solution satisfying
$$\e^2\Psi_\e''-F'(\Psi_\e)=0, \quad \qquad \lim_{x\to\pm\infty}\Psi_\e(x)=\mp1, \qquad \Psi_\e(0)=0.$$
In addition, $u^\e(h_j)=0$, for $j=1,\dots,N$ and for definiteness we choose $u^\e(a)>0$ (the case $u^\e(a)<0$ is analogous).
It is easy to check that $u^\e\in H^1(a,b)$ and satisfies \eqref{TLS1};
moreover, from \eqref{eq:energy} and \eqref{eq:translayer} it follows that
\begin{align*}
	E_\e[u^\e]&=\sum_{j=1}^{N}\int_{m_j}^{m_{j+1}}\left[\frac{\e\left[\Phi'_\e((-1)^j(x-h_j))\right]^2}{2}+\frac{F\left(\Phi_\e((-1)^j(x-h_j))\right)}{\e}\right]\,dx\\
			&=\sum_{j=1}^{N}\int_{m_j}^{m_{j+1}}|\Phi'_\e((-1)^j(x-h_j)|\sqrt{2F\left(\Phi_\e((-1)^j(x-h_j))\right)} \, dx,
\end{align*}
in view of \eqref{eq:Fi}. 
Indeed, multiplying by $\Phi_\e'$ the ODE in \eqref{eq:Fi} and using the limiting behavior, one obtains
$$\frac{\e^2}{2}(\Phi'_\e)^2=F(\Phi_\e), \qquad \qquad \mbox{ in } \R.$$
Thus, we end up with 
$$E_\e[u^\e]<\sum_{j=1}^{N}\int_{-1}^{1}\sqrt{2F(s)}\,ds<Nc_2,$$
where $c_2$ is defined in \eqref{eq:c_0}.
In conclusion, we proved that $u^\e$ defined in \eqref{eq:translayer} satisfies \eqref{E_p(w)<Nc0+eps^l}-\eqref{TLS1}.

\section{Slow motion for the time-fractional Allen--Cahn equation}\label{sec:slow}		
In this section, we study the long time dynamics of the solutions to the time-fractional Allen--Cahn problem \eqref{eq:C-model}-\eqref{eq:Neu}-\eqref{eq:initial}, 
with a particular interest to the role played by the Caputo derivative ${}^C\partial_t^\alpha u$.   
In the following subsections, we will go through the different rates of convergence towards an asymptotic configuration, 
which drastically depend on the parameter $\theta$ of the potential $F$, defined in \eqref{ipoF1} and \eqref{ipoF2}. 
Indeed, $\theta$ is the key parameter: when $\theta=2$ the solutions to \eqref{eq:C-model}-\eqref{eq:Neu}-\eqref{eq:initial} exhibit a metastable behavior, 
namely, they remain close to the profile of the initial layered datum for an exponentially long time, dependent on the parameter $\e$ and $\alpha$.
On the other hand, when $\theta>2$ the time is only algebraically long, still dependent on $\e$ and $\alpha$.
These two cases are respectively treated in the following Theorems \ref{thm:main} and  \ref{thm:main2}, 
which contain the main results of our study in the linear diffusion case.   

Our analysis is based on a layered structure choice of the initial datum. Hence, we fix a function $v$ as in \eqref{vstruct} and, throughout the paper,
 we assume that the initial datum in \eqref{eq:initial} depends on $\e$ and satisfies
\begin{equation}\label{eq:ass-u0}
	\lim_{\varepsilon\rightarrow 0^+} \|u^\varepsilon_0-v\|_{{}_{L^1}}=0.
\end{equation}
In addition, we assume that there exist  $C, \e_0>0$ such that, for any $\e\in(0,\e_0)$,
\begin{equation}\label{eq:energy-ini}
	E_\varepsilon[u^\varepsilon_0]\leq Nc_2+ C R_{\theta,\e},
\end{equation}
where $c_2$ is the constant defined in \eqref{eq:c_0} and
 $R_{\theta,\e}$ is the reminder defined in \eqref{def:Rte}. Specifically,  $u^\e_0$ satisfies the assumptions of Corollary \ref{cor:TLS}.
We recall that in \eqref{eq:translayer} we defined 
 a family of functions that satisfies \eqref{eq:ass-u0} and \eqref{eq:energy-ini}.

\subsection{Exponentially slow motion}\label{sec:exp_met}
In this subsection we show that, when the potential $F$ satisfies \eqref{ipoF2} with $\theta=2$,
 the classical Allen--Cahn metastability phenomenon is inherited by the time-fractional Allen--Cahn system \eqref{eq:C-model}-\eqref{eq:Neu}-\eqref{eq:initial}. Moreover, we show how the Caputo generalization constitutes an additional reinforcement of the metastable dynamics.  
In particular, in the following theorem, we prove that the solutions remain close to the initial layered datum for an
exponentially long time with respect to $\e$ and $\alpha$, that is, at least for a time equal to $m\exp(A/\alpha\e)$, for any $m>0$.   
\begin{thm}[exponentially slow motion when $\theta=2$]\label{thm:main}
Assume that $F\in C^1(\R)$ satisfies \eqref{ipoF1}-\eqref{ipoF2} with $\theta=2$. 
Let $v$ be as in \eqref{vstruct} and $A\in(0,r\sqrt{2\lambda_1})$, with $\lambda_1>0$ (independent on $\e$) as in \eqref{ipoF2}.
If $u^\varepsilon$ is the solution of \eqref{eq:C-model}-\eqref{eq:Neu}-\eqref{eq:initial}
with initial datum $u_0^{\varepsilon}$ satisfying \eqref{eq:ass-u0} and \eqref{eq:energy-ini} with $R_{\theta,\e}=  \exp(-A/\e)$, then
\begin{equation}\label{eq:limit}
	\sup_{0\leq t\leq m\exp\left(
		\frac{A}{\alpha \e}
		\right)}\|u^\varepsilon(\cdot,t)-v\|_{{}_{L^1}}\xrightarrow[\varepsilon\rightarrow0]{}0,
\end{equation}
for any $m>0$.
\end{thm}

In proving Theorem \ref{thm:main}, we will use an estimate of the $L^2$-norm of the fractional time-derivative of the solution $u^\e(\cdot,t)$ to \eqref{eq:C-model}-\eqref{eq:Neu}-\eqref{eq:initial}, under the same assumptions of Theorem \ref{thm:main}.  This result, presented in the following proposition, is crucial in the proof of Theorem \ref{thm:main}.

\begin{prop}\label{prop:L2-norm}
Under the same assumptions of Theorem \ref{thm:main}, there exist positive constants $\varepsilon_0$, $C$, $C_1>0$ (independent on $\varepsilon$ and $\alpha$) such that
\begin{equation}\label{L2-norm}
	\int_0^{T_{\e,\alpha}}\!\!\int_a^b \left[
		\prescript{C}{}\partial_t^{(1+\alpha)/2} u(x,t)
	\right]^2 dx dt
	\leq \frac{C\e}{\sin(\alpha\pi/2)}\exp\left(-\frac{A}\varepsilon\right),
	\end{equation}
where 
\begin{equation}\label{eq:T_ealpha}
	{T_{\e,\alpha}}:=\left[C_1
	\frac{ \alpha }{\e}\sin\left(\frac{\alpha\pi}{2}\right)
	\right]^{1/\alpha}
	\exp\left(\frac{A}{\alpha \e}\right),
	\end{equation}
for all $\varepsilon\in(0,\varepsilon_0)$.
\end{prop}

\begin{proof}
Consider $\varepsilon_0>0$ so small that for all $\varepsilon\in(0,\varepsilon_0)$, \eqref{eq:energy-ini} holds and 
\begin{equation}\label{1/2delta}
	\|u_0^\varepsilon-v\|_{{}_{L^1}}\leq\frac12\delta,
\end{equation}
where $\delta$ is the constant of Proposition \ref{prop:lower}.  
Now, let $\hat T>0$. If
\begin{equation}\label{claim1}
\frac{1}{\Gamma\left(\frac{1+\alpha}{2}\right)}
	\int_0^{\hat T}\!\!\int_a^b \left|
		\prescript{C}{} \partial_t^{(1+\alpha)/2} u(x,t)   \right|(\hat T -t)^{(\alpha-1)/2}
	 dx dt
	\leq
	\frac12\delta,
\end{equation}
then we claim that there exists a constant $C>0$ such that
\begin{equation}\label{claim2}
	E_\e[u^\e](\hat T) \geq N c_2 - C\exp\left(-\frac{A}{\e}\right).
\end{equation}
In particular, if the condition $\|u^\e(\cdot,\hat T) - v\|_{L^1}\leq \delta$ is satisfied, then \eqref{claim2} follows from Proposition \ref{prop:lower}. 
Since
$$
u^\e(x, t) - u^\e(x,0)  = 
\frac{1}{\Gamma\left(\frac{1+\alpha}{2}\right)}
\int_0^{t}
\prescript{C}{}\partial_s^{(1+\alpha)/2} u(x,s)   ( t - s)^{(\alpha-1)/2}
ds,
$$ 
see e.g. \cite[Theorem 3.8]{Diethelm}, this condition can be proven by first applying the triangular inequality, and then \eqref{1/2delta} with \eqref{claim1}. Indeed,
\begin{align*}
\| u^\e(\cdot,\hat T) - v \|_{L^1} 
& \leq 
\| u^\e(\cdot,\hat T) - u^\e_0 \|_{L^1} + \|  u^\e_0 - v \|_{L^1} 
\\ & \leq
\frac{1}{\Gamma\left(\frac{1+\alpha}{2}\right)}
	\int_0^{\hat T}\!\!\int_a^b \left|
		\prescript{C}{}\partial_t^{(1+\alpha)/2} u(x,t)   \right|(\hat T -t)^{(\alpha-1)/2}
	 dx dt
+    \frac12\delta
\leq \delta.
\end{align*}

Now, by replacing \eqref{eq:energy-ini} (with $R_{\theta,\e}$ defined by \eqref{def:Rte} for $\theta=2$) and \eqref{claim2} in equation \eqref{lem:energy} we obtain 
\begin{equation*}
		  \e^{-1}\sin(\alpha\pi/2)\int_0^{\hat T}\!\!\int_a^b\left[\prescript{C}{}\partial_t^{{(1+\alpha)}/{2}}u(x,t)\right]^2\,dxdt \leq E_\e[u](0)-E_\e[u](\hat T)
		 \leq C\exp\left(-\frac{A}{\e}\right),
\end{equation*}
with $C$ positive constant. Equivalently,
\begin{align}\label{eq:bound_on_L2_norm_Caputo_derivative}
	\int_0^{\hat{T}}\!\!\int_a^b\left[\prescript{C}{}\partial_t^{{(1+\alpha)}/{2}}u(x,t)\right]^2\,dxdt
	\leq \frac{C\e}{\sin(\alpha\pi/2)}\exp\left(-\frac{A}{\e}\right).
\end{align}
It only remains to prove that inequality \eqref{claim1} holds for $\hat T\geq T_{\e,\alpha}$.  
Since \eqref{claim1} is trivially verified for $\hat T = 0$, let us assume the existence of a value $T_*>0$ such that
\begin{equation*}
	\frac{1}{\Gamma\left(\frac{1+\alpha}{2}\right)}
	\int_0^{T_*}\!\!\int_a^b \left|
	\prescript{C}{}\partial_t^{(1+\alpha)/2} u(x,t)   \right|(T_* -t)^{(\alpha-1)/2}
	dx dt=\frac12\delta.
\end{equation*}
Then, by applying H\"older's inequality in the above and \eqref{eq:bound_on_L2_norm_Caputo_derivative}, we have
\begin{align*}
	\frac{1}{2}\delta &\leq
	\frac{1}{\Gamma\left(\frac{1+\alpha}{2}\right)}
	\left[\int_0^{T_*}\!\!\int_a^b
	\left(
	\prescript{C}{}\partial_t^{(1+\alpha)/2} u(x,t)   \right)^2
	dxdt
	\right]^{1/2}
	\left[\int_0^{T_*}\!\!\int_a^b
	(T_* -t)^{(\alpha-1)}
		dxdt
	\right]^{1/2}
	\\
	&\leq 
	\frac{1}{\Gamma\left(\frac{1+\alpha}{2}\right)}
	\left[\frac{C\e}{\sin(\alpha\pi/2)}\exp\left(-\frac{A}{\e}\right)
	\frac{ (b-a)T_*^\alpha}{\alpha}\right]^{1/2},
\end{align*}
so that 
\begin{align*}
	{T_*}
	\geq 
	\left[
	C_1
	\frac{ {\alpha}}{\e}\sin(\alpha\pi/2)
	\right]^{1/\alpha}
	\exp\left(\frac{A}{\alpha \e}\right),
\end{align*}
for some $C_1>0$, independent on $\alpha$ and $\e$. 

On the other hand, if such a $T_*$ does not exist, then  \eqref{claim1} holds for all $\hat T$. Therefore the proof is complete.
\end{proof}

Thanks to this result, we can now prove Theorem \ref{thm:main}.
\begin{proof}[Proof of Theorem \ref{thm:main}]
Fix $m>0$.
As a consequence of the triangle inequality we have that
\begin{equation}\label{trianglebar}
	\sup_{0\leq t\leq m\exp\left(
		\frac{A}{\alpha \e}
		\right)}\|u^\varepsilon(\cdot,t)-v\|_{{}_{L^1}}\leq \sup_{0\leq t\leq m\exp\left(
		\frac{A}{\alpha \e}
		\right)}\|u^\varepsilon(\cdot,t)-u_0^\varepsilon\|_{{}_{L^1}}+\|u_0^\varepsilon-v\|_{{}_{L^1}}.
\end{equation}
Hence, since the second term in the right-hand side of  \eqref{trianglebar} tends to $0$ by \eqref{eq:ass-u0}, in order to prove \eqref{eq:limit}, it is sufficient to show that 
\begin{equation}\label{primopezzo}
	\sup_{0\leq t\leq m\exp\left(
		\frac{A}{\alpha \e}
		\right)}\|u^\varepsilon(\cdot,t)-u_0^\e\|_{{}_{L^1}}\xrightarrow[\varepsilon\rightarrow0]{}0.
\end{equation}
To this aim, we first observe that up to taking $\e>0$ so small that 
$$m<\left[C_1\alpha\sin(\alpha\pi/2)\right]^{1/\alpha}\e^{-1/\alpha},$$
we can apply \eqref{L2-norm} to deduce
\begin{equation}\label{proof:usata}
	\int_0^{m\exp\left(
		\frac{A}{\alpha \e}
		\right)}\!\!\int_a^b \left[
	\prescript{C}{}\partial_t^{(1+\alpha)/2} u(x,t)
	\right]^2 dx dt
	\leq \frac{C\e}{\sin(\alpha\pi/2)}\exp\left(-\frac{A}\varepsilon\right).
\end{equation}
Moreover, for all $t\in\left[0,{m\exp\left(
	\frac{A}{\alpha \e}
	\right)}\right]$, by applying the H\"older's inequality and \eqref{proof:usata}, we have
\begin{align*}
	\|u^\e(\cdot,t)-u^\e_0\|_{{}_{L^1}}
	&\leq\frac{1}{\Gamma\left(\frac{1+\alpha}{2}\right)}
	\int_0^{t}\!\!\int_a^b \left|\prescript{C}{}\partial_s^{(1+\alpha)/2} u(x,s) \right|\left(t -s \right)^{(\alpha-1)/2}dxds\\
	&\leq
	\frac{1}{\Gamma\left(\frac{1+\alpha}{2}\right)}
	\left[\int_0^{t}\!\!\int_a^b \left|\prescript{C}{}\partial_s^{(1+\alpha)/2} u(x,s)\right|^2 dxds\right]^{1/2}
	\left[
	\int_0^{t}\!\!\int_a^b 
	\left(t -s \right)^{\alpha-1}dxds\right]^{1/2}\\
	&\leq\frac{1}{\Gamma\left(\frac{1+\alpha}{2}\right)}
	\sqrt{\frac{C\e}{\sin(\alpha\pi/2)}}\exp\left(-\frac{A}{2\e}\right)
	\sqrt{m^\alpha (b-a)}\exp\left(\frac{A \alpha}{2\alpha\e}\right)
	\leq C_\alpha \sqrt{\e},
\end{align*}
being $C_\alpha$ a positive constant depending on $\alpha$ but independent of $\e$.
We thus obtained \eqref{primopezzo} and the proof is complete.
\end{proof}

\begin{rem}
Theorem \ref{thm:main} describes very well the role of the fractional order $\alpha\in(0,1)$ for long times: 
the fractional derivative has an extra slowing effect on the metastable dynamics of the classical Allen--Cahn equation (obtained by formally setting $\alpha=1$), 
since the time taken for the solution to lose its initial transition structure shifts from $\exp(A/\e)$ to $\exp(A/\alpha\e)$.
\end{rem}

\subsection{Algebraic slow motion}\label{sec:alg-lin}
When the potential $F$ satisfies assumption \eqref{ipoF2} with $\theta>2$, the solutions to \eqref{eq:C-model}-\eqref{eq:Neu}-\eqref{eq:initial} evolve substantially differently from the critical case $\theta=2$, presented in the previous section. Indeed, when $\theta=2$, 
the \emph{non-degenerate potentials} give rise to exponentially slow motion, as described by
Theorem \ref{thm:main}.
Instead, in this subsection, we show that, when $\theta>2$, the solution maintains the same unstable profile of the initial datum \emph{only} for an algebraically long time with respect to $\e$, that is, at least for a time equals $ l\e^{-\beta/\alpha}$, for any $l>0$, with $\beta>0$ defined in \eqref{eq:beta}. Also in this case, the fractional order $\alpha$ contributes to a further slow down of the dynamics. The passage from exponential to algebraic slow motion results from the application of Proposition \ref{prop:lower_deg} instead of Proposition \ref{prop:lower}, and, particularly, it is related to the reminder in the energy lower bounds.
This constitutes the main result of this subsection, as described by the following theorem.

\begin{thm}[algebraic slow motion when $\theta>2$]\label{thm:main2}
Assume that $F\in C^1(\R)$ satisfies \eqref{ipoF1}-\eqref{ipoF2} with $\theta>2$. 
Moreover, let $v$ be as in \eqref{vstruct} and let $\left\{k_j\right\}_{j\in\mathbb N}$ be as in \eqref{eq:exp_alg}. 
If $u^\varepsilon$ is the solution to \eqref{eq:C-model}-\eqref{eq:Neu}-\eqref{eq:initial},
with initial profile $u_0^{\varepsilon}$ satisfying \eqref{eq:ass-u0} and \eqref{eq:energy-ini} with $R_{\theta,\e}=\e^{k_j}$, then 
\begin{equation}\label{eq:limit-deg}
	\sup_{0\leq t\leq l{\e^{-k_j/\alpha}}}\|u^\varepsilon(\cdot,t)-v\|_{{}_{L^1}}\xrightarrow[\varepsilon\rightarrow0]{}0,
\end{equation}
for any $l>0$.
\end{thm}
The proof of Theorem \ref{thm:main2} follows the same steps of the proof of Theorem \ref{thm:main}. In particular, it uses Proposition \ref{prop:lower_deg} and the following Proposition \ref{prop:L2-norm_algebraic}, instead of, respectively, Proposition \ref{prop:lower} and \ref{prop:L2-norm}.
Indeed, similarly to Proposition \ref{prop:L2-norm}, when $\theta>2$ in the definition  \eqref{ipoF1}-\eqref{ipoF2}  of $F$, one can prove the following estimate.
\begin{prop}\label{prop:L2-norm_algebraic}
Under the same assumptions of Theorem \ref{thm:main2}, there exist positive constants $\varepsilon_0$, $C$, $C_j>0$ (independent on $\varepsilon$ and $\alpha$) such that
\begin{equation}\label{L2-norm_algebraic}
	\int_0^{T_{\e,\alpha}}
	\int_a^b \left[
	\prescript{C}{}\partial_t^{(1+\alpha)/2} u(x,t)
	\right]^2 dx dt
	\leq \frac{C}{\sin(\alpha\pi/2)}\e^{k_j+1},
\end{equation}
where 
\begin{equation}\label{eq:T_ealpha_algebraic}
	{T_{\e,\alpha}}
	:= 
	\left[
	C_j 
	\alpha
	\sin(\alpha\pi/2)
	\right]^{1/\alpha}
	\e^{-(k_j+1)/\alpha},
\end{equation}
for all $\varepsilon\in(0,\varepsilon_0)$.
\end{prop}

\subsection{Layer Dynamics}
In this subsection, we conclude our investigation on the time-fractional Allen--Cahn model \eqref{eq:C-model}-\eqref{eq:Neu}-\eqref{eq:initial}, describing the slow motion of the transition points $h_1,\ldots,h_N$, both in the critical  and supercritical case, respectively $\theta=2$ and $\theta>2$. Here, our main result is described by the following Theorem \ref{thm:interface}, which shows  that the transition layers evolve with a velocity which goes to zero as $\e\to0^+$.  

Before presenting the theorem, let us fix a function $v$ as in \eqref{vstruct} and define its {\it interface} $I[v]$ as the set
\begin{equation*}
	I[v]:=\{h_1,h_2,\ldots,h_N\}.
\end{equation*}
More generally, for any function $u:[a,b]\rightarrow\mathbb{R}$ and for any closed subset $K\subset(-1,1)$, we define
the {\it interface} $I_K[u]$ by
\begin{equation*}
	I_K[u]:=u^{-1}(K).
\end{equation*}

The following lemma establishes that, if a function $u\in H^1([a,b])$ is close to $v$ in the $L^1$-norm and its energy is greater than the minimum energy defined in Corollary \ref{cor:TLS} associated to the $N$ transitions, then the distance between the interfaces of $u$ and $v$ remains small. Here the distance is the {\it Hausdorff distance} is defined by
\begin{equation*}
	d(A,B):=\max\biggl\{\sup_{\alpha\in A}d(\alpha,B),\,\sup_{\beta\in B}d(\beta,A)\biggr\},
\end{equation*}
where $d(\beta,A):=\inf\{|\beta-\alpha|: \alpha\in A\}$, for every $\beta\in B $, where $A, B\subseteq\R$.

\begin{lem}\label{lem:interface}
Assume that $F\in C^1(\R)$ satisfies \eqref{ipoF1}-\eqref{ipoF2} with $\theta\geq 2$ 
and let $v$ be as in \eqref{vstruct}.
Given $\delta_1\in(0,r)$ and a closed subset $K\subset(-1,1)$, 
there exist positive constants $\hat\delta,\varepsilon_0$  such that, if for all $\varepsilon\in(0,\varepsilon_0)$  $u\in H^1([a,b])$ satisfies
\begin{equation}\label{eq:lem-interf}
	\|u-v\|_{{}_{L^1}}<\hat\delta \qquad \quad \mbox{ and } \qquad \quad E_\varepsilon[u]\leq Nc_2+M_\e,
\end{equation}
for some $M_\e>0$ and with $E_\e[u]$ defined in \eqref{eq:energy}, we have
\begin{equation}\label{lem:d-interfaces}
	d(I_K[u], I[v])<\tfrac12\delta_1.
\end{equation}
\end{lem}

\begin{proof}
 The proof can be found in \cite[Lemma 4.5]{DLFS}.
\end{proof}

The above lemma, together with Theorems \ref{thm:main}, \ref{thm:main2} allows us to prove the following result. 
Following \cite{Bron-Kohn,DLFS,FPS-DCDS,Grant}, Theorem \ref{thm:interface} establishes that, in order for the distance between the interface of the time-dependent solution $u^\e(\cdot,t)$ and that of the initial datum to exceed a certain quantity, one has to wait at least an exponentially long time in $\e$ when $\theta=2$, or an algebraically long time in $\e$ in the supercritical case $\theta>2$.

\begin{thm}\label{thm:interface}
Let $F\in C^1(\R)$ satisfy \eqref{ipoF1}-\eqref{ipoF2} with {$\theta\geq 2$} and $u^\varepsilon$ be the solution of \eqref{eq:C-model}-\eqref{eq:Neu}-\eqref{eq:initial}, 
with initial datum $u_0^{\varepsilon}$ satisfying \eqref{eq:ass-u0} and \eqref{eq:energy-ini}. 
Given $\delta_1\in(0,r)$ and a closed subset $K\subset(-1,1)$, set
\begin{equation*}
	t_\varepsilon(\delta_1)=\inf\{t:\; d(I_K[u^\varepsilon(\cdot,t)],I_K[u_0^\varepsilon])>\delta_1\}.
\end{equation*}
Then, there exists $\varepsilon_0>0$ such that if $\varepsilon\in(0,\varepsilon_0)$
\begin{equation*}
	t_\varepsilon(\delta_1)> \left\{ \begin{aligned}
& \exp\left(A/(\alpha\varepsilon)\right), \quad  &\mbox{if}  \quad &\theta=2, \\
& \e^{-k_j/\alpha},   \quad  &\mbox{if}  \quad &\theta>2,
\end{aligned}\right.
\end{equation*}
where $A$ and $k_j$ are defined as in Propositions \ref{prop:lower} and \ref{prop:lower_deg}, respectively. 
\end{thm}
\begin{proof}
Since $u_0^\e$ satisfies \eqref{eq:ass-u0} and \eqref{eq:energy-ini}, one can always find a small   $\varepsilon_0>0$ such that  $u_0^\e $ automatically verifies assumption \eqref{eq:lem-interf} for all $\varepsilon\in(0,\varepsilon_0)$. 
Therefore, we can apply Lemma \ref{lem:interface}, 
obtaining that
\begin{equation}\label{interfaces-u0}
	d(I_K[u_0^\varepsilon], I[v])<\tfrac12\delta_1,
\end{equation}
 for all $\varepsilon\in(0,\varepsilon_0)$.
 
Now, let us fix $\varepsilon\in(0,\varepsilon_0)$
and consider $u^\varepsilon(\cdot,t)$ for
\begin{equation}\label{time}
t\leq\left\{ \begin{aligned}
& \exp(A/\alpha\varepsilon) \quad  &\mbox{if}  \quad &\theta=2, \\
& \e^{-k_j/\alpha}  \quad  &\mbox{if}  \quad &\theta>2.
\end{aligned}\right.
\end{equation}
Then, $u^\varepsilon(\cdot,t)$ satisfies both the conditions in \eqref{eq:lem-interf}: the first one follows from \eqref{eq:limit} when $\theta=2$ and from \eqref{eq:limit-deg} when $\theta>2$; the second condition follows from \eqref{remark:E(t)<E(0)}.
Again, we can apply Lemma \ref{lem:interface} and conclude that
\begin{equation}\label{interfaces-u}
	d(I_K[u^\varepsilon(\cdot,t)], I[v])<\tfrac12\delta_1
\end{equation}
for all $t$ as in \eqref{time}.  
Therefore, by combining \eqref{interfaces-u0} and \eqref{interfaces-u}, and using the triangle inequality, we obtain $d(I_K[u^\varepsilon(\cdot,t)],I_K[u_0^\varepsilon])<\delta_1$ for all $t$ as in \eqref{time},
which proves the theorem.
\end{proof}

\section{The nonlinear diffusion case}\label{sec:nonlinear}
\subsection{The $p$-Laplacian diffusion}\label{sec:p-lap}
Here we consider equation \eqref{eq:p-lap} with homogeneous Neumann boundary conditions \eqref{eq:Neu} and initial condition \eqref{eq:initial}.
The first step is to prove a dissipative result equivalent to \eqref{lem:energy}; 
in this case we define the energy
\begin{equation}\label{eq:energy-p}
	\mathcal{E}_{\e,p}[u]= \int_a^b\left[\frac{\e^{p-1}|u_x|^p}{p}+\frac{F(u)}{\e}\right]\,dx,
\end{equation}
and we have the following result.
\begin{lem}\label{lem:energy-p}
Let $u$ be a sufficiently regular solution of \eqref{eq:p-lap}-\eqref{eq:Neu}-\eqref{eq:initial}.
If $\mathcal{E}_{\e,p}$ is the functional defined in \eqref{eq:energy-p} then 
\begin{equation}\label{eq:energy-dec}
	\mathcal{E}_{\e,p}[u](0)-\mathcal{E}_{\e,p}[u](T)\geq \e^{-1}\sin\left(\frac{\alpha\pi}{2}\right)\int_0^T\!\!\int_a^b\left[\prescript{C}{}\partial_t^{(1+\alpha)/2}u(x,t)\right]^2\,dxdt\geq0,
\end{equation}
for any $\alpha\in(0,1)$.
\end{lem}
\begin{proof}
Proceeding as in the proof of \eqref{eq:energyestimate}, and multiplying \eqref{eq:p-lap} by $u_t$ and integrating over the spatial domain $(a,b)$, we obtain 
 $$
\int_{a}^{b} 
\prescript{C}{}\partial_t^{\alpha}u(x,t) u_t(x,t) \, dx = \int_{a}^{b} \left[\e^p(|u_x|^{p-2}u_x)_x-F'(u)\right]u_t(x,t)\,dx.
$$
By using integration by parts on the right-hand side and the boundary Neumann conditions \eqref{eq:Neu}, we deduce (see also \cite[Lemma 3.1]{FPS-DCDS})
$$
 \int_{a}^{b} \left[\e^p(|u_x|^{p-2}u_x)_x-F'(u)\right]u_t(x,t)\,dx=-\e\frac{d}{dt}\mathcal{E}_{\e,p}[u](t).
$$
Thus, integrating with respect to time, one has
$$
\e^{-1}  \int_{0}^{T}\!\int_{a}^{b}
\prescript{C}{}\partial_t^{\alpha}u(x,t) u_t(x,t) \, dx \, dt = \mathcal{E}_{\e,p}[u](0) - \mathcal{E}_{\e,p}[u](T).
$$
By using \eqref{eq:derivatives_lucho}, we obtain exactly \eqref{eq:energy-dec}.
\end{proof}
We now recall the lower bound on the energy proved in \cite[Proposition 3.2]{FPS-DCDS}; define
\begin{equation}\label{eq:c_p}
	c_p:=\left(\frac{p}{p-1}\right)^{\frac{p-1}{p}}\int_{-1}^{+1} {F(s)}^\frac{p-1}{p}\,ds, \qquad p>1,
\end{equation}
and notice that for $p=2$ we obtain the constant defined in \eqref{eq:c_0}.
Moreover, define 
\begin{equation}\label{eq:mu}
	\mu_p:=\begin{cases}
		\displaystyle\sqrt2\left(\frac{\lambda_1}{p-1}\right)^{\frac{1}{p}}, \qquad & 1<p\leq2,\\
		\displaystyle\frac{\sqrt{2\lambda_1}\lambda_2^{\frac{2-p}{2p}}}{(p-1)^{{\frac{1}{p}}}}, &p>2,
	\end{cases}
\end{equation}
where $\lambda_1$ and $\lambda_2$ are the constants appearing in \eqref{ipoF2}.
\begin{prop}\label{prop:lower-p}
Let $p>1$ and assume that $F\in C^1(\R)$ satisfies \eqref{ipoF1}-\eqref{ipoF2} with $\theta=p$. 
Let $v$ be as in \eqref{vstruct}, and fix $A\in(0,r\mu_p)$, 
where $\mu_p$ is defined in \eqref{eq:mu}.
Then there exist $\e_0,C,\delta>0$ (depending only on $p,v$ and $A$) such that if $u\in H^1(a,b)$ satisfies \eqref{eq:u-v},
then for any $\e\in(0,\e_0)$,
\begin{equation}\label{eq:lower-p}
	\mathcal{E}_{\e,p}[u]\geq Nc_p-C\exp(-Ap/2\varepsilon),
\end{equation}
where $\mathcal{E}_{\e,p}$ and $c_p$ are defined in \eqref{eq:energy-p} and \eqref{eq:c_p}, respectively.
\end{prop}

\begin{proof}
This result has been proved in \cite[Proposition 3.2]{FPS-DCDS} in the explicit case 
$$F(u)=\frac{1}{2p}|1-u^2|^p, \quad p>1.$$
Here, we only show how to extend the result to a general $F$ satisfying \eqref{ipoF1}-\eqref{ipoF2} with $\theta=p$. 
The explicit form for $F$ is used in the estimate of $\psi(x):=(z(x)-1)^2$, where $z$ is the minimizer defined in \cite[pag. 3226]{FPS-DCDS}.
Let us see how to proceed in the general case; one has $\psi'=2(z-1)z'$ and 
\begin{equation*}
	\psi''(x)\geq \frac{2}{\e^{2-p}}\left( \frac{p}{p-1}\right)^{{\frac{2-p}{p}}}(z(x)-1)\frac{F'(z(x))}{\e^p(p-1)} F(z(x))^{\frac{2-p}{p}},
\end{equation*}
see \cite[pag. 3226-3227]{FPS-DCDS}. 
We must consider two different cases: if $p\in(1,2]$, from \eqref{ipoF2}-\eqref{eq:ass-F3} with $\theta=p$ it follows that
\begin{equation*}
\begin{aligned}
	\psi''(x) & \geq \frac{2}{\e^2} \lambda_1^{\frac{2}{p}}(p-1)^{-{\frac{2}{p}}}(z(x)-1)^2 \\
	&\geq \frac{2}{\e^2} \frac{\lambda_1^{\frac{2}{p}}}{(p-1)^{\frac{2}{p}}} (z(x)-1)^2\\
	& = \frac{\mu_p^2}{\e^2}\psi(x), \qquad \qquad 1<p\leq2.
\end{aligned}
\end{equation*}
On the other hand, for $p>2$ \eqref{ipoF2}-\eqref{eq:ass-F3} and $\theta=p$ give
\begin{equation*}
	\psi''(x)\geq \frac{2}{\e^2}\lambda_1\lambda_2^{\frac{2-p}{p}}(p-1)^{-{\frac{2}{p}}}(z(x)-1)^2=\frac{\mu_p^2}{\e^2}\psi(x), \qquad \qquad p>2.
\end{equation*}
The latter computation is the only difference with \cite{FPS-DCDS} (it is the only place where the explicit form of $F$ is used in \cite{FPS-DCDS}),
and the proof can be concluded in the same way.
\end{proof}

Thanks to Proposition \ref{prop:lower-p}, we can proceed as in Section \ref{sec:exp_met} to prove that a transition layer structure is maintained for an exponentially long time.
Let us assume that the initial datum in \eqref{eq:initial} depends on $\e$, it satisfies \eqref{eq:ass-u0} and that
there exist constants $C>0$, $A\in(0,r\mu_p)$ 
(with $r$ satisfying \eqref{vstruct} and $\mu_p$ defined in \eqref{eq:mu}) such that
\begin{equation}\label{eq:energy-ini-p}
	\mathcal{E}_{\e,p}[u^\varepsilon]\leq Nc_p+C\exp(-Ap/2\e),
\end{equation}
for any $\varepsilon\ll1$, where the energy $\mathcal{E}_{\e,p}$ and the positive constant $c_ p$ are defined in \eqref{eq:energy-p} and in \eqref{eq:c_p}, respectively.
Then, we have the following result.

\begin{thm}[exponentially slow motion with a $p$-Laplacian diffusion when $\theta=p$]\label{thm:main-p}
Assume that $F\in C^1(\R)$ satisfies \eqref{ipoF1}-\eqref{ipoF2} with $\theta=p>1$. 
Let $v$ be as in \eqref{vstruct} and $A\in(0,r\mu_p)$, with $\mu_p>0$ (independent on $\e$) as in \eqref{eq:mu}.
If $u^\varepsilon$ is the solution of \eqref{eq:p-lap}-\eqref{eq:Neu}-\eqref{eq:initial}
with initial datum $u_0^{\varepsilon}$ satisfying \eqref{eq:ass-u0} and \eqref{eq:energy-ini-p}, then
\begin{equation}\label{eq:limit-p}
	\sup_{0\leq t\leq m\exp\left(
		\frac{Ap}{2\alpha \e}
		\right)}\|u^\varepsilon(\cdot,t)-v\|_{{}_{L^1}}\xrightarrow[\varepsilon\rightarrow0]{}0,
\end{equation}
for any $m>0$.
\end{thm}

\begin{proof}
By proceeding as in the proof of Proposition \ref{prop:L2-norm} and using \eqref{eq:lower-p} instead of \eqref{eq:lower}, 
we can prove that there exist positive constants $\varepsilon_0$, $C$, $C_1>0$ (independent on $\varepsilon$ and $\alpha$) such that
\begin{equation*}
	\int_0^{T_{\e,\alpha,p}}\!\!\int_a^b \left[
		\prescript{C}{}\partial_t^{(1+\alpha)/2} u(x,t)
	\right]^2 dx dt
	\leq \frac{C\e}{\sin(\alpha\pi/2)}\exp\left(-\frac{Ap}{2\varepsilon}\right),
	\end{equation*}
where 
\begin{equation*}
	{T_{\e,\alpha,p}}:=\left[C_1
	\frac{ \alpha }{\e}\sin\left(\frac{\alpha\pi}{2}\right)
	\right]^{1/\alpha}
	\exp\left(\frac{Ap}{2\alpha\e}\right),
	\end{equation*}
for all $\varepsilon\in(0,\varepsilon_0)$.
Indeed, reasoning as we did for \eqref{eq:bound_on_L2_norm_Caputo_derivative} and applying \eqref{eq:energy-dec}-\eqref{eq:lower-p}-\eqref{eq:energy-ini-p}, 
one can prove that there exists $\hat T>0$ such that
\begin{equation}\label{eq:T-plap}
	\int_0^{\hat{T}}\!\!\int_a^b\left[\prescript{C}{}\partial_t^{{(1+\alpha)}/{2}}u(x,t)\right]^2\,dxdt
	\leq \frac{C\e}{\sin(\alpha\pi/2)}\exp\left(-\frac{Ap}{2\e}\right).
\end{equation}
It only remains to prove that \eqref{eq:T-plap} holds for $\hat T\geq T_{\e,\alpha,p}$.
This can be done by using \eqref{claim1}, H\"older's inequality and \eqref{eq:T-plap} to get
\begin{align*}
	\frac{1}{2}\delta &\leq
	\frac{1}{\Gamma\left(\frac{1+\alpha}{2}\right)}
	\left[\int_0^{\hat T}\!\!\int_a^b
	\left(
	\prescript{C}{}\partial_t^{(1+\alpha)/2} u(x,t)   \right)^2
	dxdt
	\right]^{1/2}
	\left[\int_0^{\hat T}\!\!\int_a^b
	(\hat T -t)^{(\alpha-1)}
		dxdt
	\right]^{1/2}
	\\
	&\leq 
	\frac{1}{\Gamma\left(\frac{1+\alpha}{2}\right)}
	\left[\frac{C\e}{\sin(\alpha\pi/2)}\exp\left(-\frac{Ap}{2\e}\right)
	\frac{ (b-a)\hat T^\alpha}{\alpha}\right]^{1/2}.
\end{align*}
As a consequence, 
\begin{align*}
	\hat T
	\geq 
	\left[C_1
	\frac{ \alpha }{\e}\sin\left(\frac{\alpha\pi}{2}\right)
	\right]^{1/\alpha}
	\exp\left(\frac{Ap}{2\alpha\e}\right)
	=T_{\e,\alpha,p},
\end{align*}
for some $C_1>0$, independent on $\alpha$ and $\e$. 

Therefore, we can conclude the proof exactly as done for Theorem \ref{thm:main}.
\end{proof}

The argument developed in Section \ref{sec:alg-lin} can be adapted to the degenerate case $\theta>p$. 
In this case, the lower bound for the energy is algebraic rather than exponentially small, as in the classical setting considered in \cite{FPS-DCDS}. 
Combining this estimate with the fractional energy dissipation inequality gives the corresponding algebraic time scale, with the exponent divided by $\alpha$.

Let us define the sequence
\begin{equation}\label{eq:exp_alg_p}
	\varkappa_j:=\sum_{m=1}^j\gamma_p^m,
	\qquad\qquad  \mbox{ where }  \quad \gamma_p:=\displaystyle\frac{p-1}{p}+\frac{1}{\theta},
\end{equation}
\begin{thm}[algebraic slow motion with a $p$-Laplacian diffusion when $\theta>p$]\label{thm:alg-p}
Assume that $F\in C^1(\R)$ satisfies \eqref{ipoF1}-\eqref{ipoF2} with $\theta>p>1$.
Moreover, let $v$ be as in \eqref{vstruct} and let $\left\{\varkappa_j\right\}_{j\in\mathbb N}$ be as in \eqref{eq:exp_alg_p}. 
If $u^\varepsilon$ is the solution of \eqref{eq:p-lap}-\eqref{eq:Neu}-\eqref{eq:initial}
with initial datum $u_0^{\varepsilon}$ satisfying \eqref{eq:ass-u0} and 
\begin{equation*}
	\mathcal E_{\varepsilon,p}[u_0^\varepsilon]\leq Nc_p+\e^{\varkappa_j},
\end{equation*}
for any $\varepsilon\ll1$, then
\begin{equation*}
	\sup_{0\leq t\leq l{\e^{-\varkappa_j/\alpha}}}\|u^\varepsilon(\cdot,t)-v\|_{{}_{L^1}}\xrightarrow[\varepsilon\rightarrow0]{}0,
\end{equation*}
 for any $l>0$.
\end{thm}

\begin{rem}
Notice that Theorems \ref{thm:main} and \ref{thm:main2} can be recovered from Theorems \ref{thm:main-p} and \ref{thm:alg-p}, respectively.
In particular, the energy \eqref{eq:energy-p} coincides with \eqref{eq:energy} when $p=2$; 
moreover, the constants $\mu_p$, defined in \eqref{eq:mu}, and $\gamma_p$ in \eqref{eq:exp_alg_p} become $\sqrt{2\lambda_1}$ and \eqref{eq:exp_alg}, 
respectively, when $p=2$.
\end{rem}

\subsection{Perona--Malik type diffusion}
We finally consider a class of nonlinear diffusion operators of Perona--Malik type. 
Let $Q\in C^1(\R)$ satisfy \eqref{eq:Q-ass1}-\eqref{eq:Q-ass2}, and consider the problem \eqref{eq:Q-dif},
complemented with the homogeneous Neumann boundary conditions \eqref{eq:Neu} and the initial condition \eqref{eq:initial}. 
The corresponding energy is given by
\begin{equation}\label{eq:energy-Q}
	\mathscr{E}_\e[u]=\int_a^b\left(\frac{1}{\e^3}\tilde{Q}(\e^2u_x)+\frac{F(u)}{\e}\right)\,dx, \qquad \mbox{ where } \quad \tilde Q(s):=\int_0^sQ(\tau)\,d\tau.
\end{equation}

The slow motion of layered solutions for the corresponding problem with classical time derivative has been studied in \cite{DLFS}. 
In particular, in the non-degenerate case $\theta=2$, the energy of a function sufficiently close to a layered profile satisfies the following lower bound \eqref{eq:lower-Q}.
Let us first introduce the energy of a single transition layer; set
\begin{equation*}
	P_\e(s):= \int_{0}^{s}\e^2z\,Q'(\e^2 z)\,dz,
\end{equation*}
and denote by $J_\e$ the function which inverts the equation $P_\e(s)=\xi$ in $[0,\kappa \e^{-2}]$,
for any $s\in[0,\kappa \e^{-2}]$ and $\xi\in[0,\ell\e^{-2}]$, for details see \cite[Lemma 2.1]{DLFS}.
Then, the energy of a single transition layer is given by
\begin{equation}\label{eq:c_eps}
	c_\e:=\e^{-1}\int_{-1}^{+1}Q\left(\e^2 J_\e(F(s))\right)\,ds.
\end{equation}
\begin{prop}[\cite{DLFS}]\label{prop:lower-Q}
Assume that $Q\in C^1(\R)$ satisfies \eqref{eq:Q-ass1}-\eqref{eq:Q-ass2} and that $F\in C^1(\R)$ satisfies \eqref{ipoF1}-\eqref{ipoF2} with $\theta=2$. 
Let us set 
\begin{equation}\label{eq:mathc-Q}
	\mathcal{Q}:=\max_{s\in[-\kappa,\kappa]}Q'(s),
\end{equation}
where $\kappa$ is given in \eqref{eq:Q-ass2}.
Moreover, let $v$ be as in \eqref{vstruct} and $A\in(0,r\sqrt{2\lambda_1 \mathcal{Q}^{-1}})$ with
$\lambda_1>0$ (independent on $\e$) as in \eqref{ipoF2}.
Then, there exist $\e_0,C,\delta>0$ (depending only on $Q,F,v$ and $A$) such that if $u\in H^1(a,b)$ satisfies \eqref{eq:u-v},
then for any $\e\in(0,\e_0)$,
\begin{equation}\label{eq:lower-Q}
	\mathscr E_\varepsilon[u]\geq Nc_\e-C\exp(-A/\varepsilon),
\end{equation}
where $\mathscr E_\e$ and $c_\e$ are defined in \eqref{eq:energy-Q} and \eqref{eq:c_eps}, respectively. 
\end{prop}

This estimate is the analogue of Propositions \ref{prop:lower} and \ref{prop:lower-p} used in Sections \ref{sec:exp_met} and \ref{sec:p-lap}, respectively. 
In particular, Proposition \ref{prop:lower-Q} is a variational result concerning the energy \eqref{eq:energy-Q}, 
and the exponentially small remainder does not depend on the order $\alpha$ of the fractional derivative.

Combining this lower bound with the fractional energy dissipation estimate, we obtain the persistence of the layered structure on an exponentially long time scale. 

\begin{thm}[exponentially slow motion for Perona--Malik type diffusion]\label{thm:Q-main}
Assume that $Q\in C^1(\R)$ satisfies \eqref{eq:Q-ass1}-\eqref{eq:Q-ass2} and that $F\in C^1(\R)$ satisfies \eqref{ipoF1}-\eqref{ipoF2} with $\theta=2$. 
Let $v$ be as in \eqref{vstruct} and let $A\in(0,r\sqrt{2\lambda_1\mathcal{Q}^{-1}})$, with $\mathcal{Q}$ defined in \eqref{eq:mathc-Q}. 
If $u^\e$ is a sufficiently regular solution of \eqref{eq:Q-dif}-\eqref{eq:Neu}-\eqref{eq:initial} with initial datum satisfying \eqref{eq:ass-u0} and
\begin{equation*}
	\mathscr{E}_{\e}[u^\varepsilon]\leq Nc_\e+C\exp(-A/\e),
\end{equation*}
for any $\varepsilon\ll1$, where the energy $\mathscr E_{\e}$ and the positive constant $c_\e$ 
are defined in \eqref{eq:energy-Q} and \eqref{eq:c_eps}, respectively, then
\begin{equation}\label{eq:Q-limit}
	\sup_{0\leq t\leq m\exp(A/(\alpha\e))} \|u^\e(\cdot,t)-v\|_{L^1}\xrightarrow[\e\to0]{}0,
\end{equation}
for any $m>0$.
\end{thm}
\begin{proof}
The proof follows the same strategy as that of Theorems \ref{thm:main}-\ref{thm:main-p}. 
Indeed, the lower bound \eqref{eq:lower-Q} provides the exponentially small energy gap, 
while the fractional dissipation inequality yields the corresponding control of the time variation of the solution. 

First, multiplying \eqref{eq:Q-dif} by $u_t$ and proceeding as in the proof of \eqref{eq:energyestimate} and \eqref{eq:energy-dec}, we obtain the fractional dissipation inequality
\begin{equation}\label{eq:energyestimate-Q}
	\mathscr E_\e[u](0)-\mathscr E_\e[u](T)\geq \e^{-1}\sin\left(\frac{\alpha\pi}{2}\right)\int_0^T\!\int_a^b\left[\prescript{C}{}\partial_t^{(1+\alpha)/2}u(x,t)\right]^2\,dxdt\geq0.
\end{equation}
Second, we can prove that there exist positive constants $\varepsilon_0, C_1, C_2>0$ (independent on $\varepsilon$) such that
\begin{equation*}
	\int_0^{\bar T_{\e,\alpha}}\!\!\int_a^b \left[
		\prescript{C}{}\partial_t^{(1+\alpha)/2} u(x,t)
	\right]^2 dx dt
	\leq \frac{C_1\e}{\sin(\alpha\pi/2)}\exp\left(-\frac{A}{\varepsilon}\right),
	\end{equation*}
where 
\begin{equation*}
	\bar T_{\e,\alpha}:=\left[C_2
	\frac{ \alpha }{\e}\sin\left(\frac{\alpha\pi}{2}\right)
	\right]^{1/\alpha}
	\exp\left(\frac{A}{\alpha\e}\right),
	\end{equation*}
for all $\varepsilon\in(0,\varepsilon_0)$.
Indeed, reasoning as we did for \eqref{eq:bound_on_L2_norm_Caputo_derivative} and \eqref{eq:T-plap}, but here applying \eqref{eq:lower-Q}-\eqref{eq:energyestimate-Q},
one can prove that there exists $\hat T>0$ such that
\begin{equation*}
	\int_0^{\hat{T}}\!\!\int_a^b\left[\prescript{C}{}\partial_t^{{(1+\alpha)}/{2}}u(x,t)\right]^2\,dxdt
	\leq \frac{C_1\e}{\sin(\alpha\pi/2)}\exp\left(-\frac{A}{\e}\right),
\end{equation*}
and, by using H\"older's inequality, we end up with
\begin{equation*}
	\hat T\geq \left[C_2
	\frac{ \alpha }{\e}\sin\left(\frac{\alpha\pi}{2}\right)
	\right]^{1/\alpha}
	\exp\left(\frac{A}{\alpha\e}\right)
	=\bar T_{\e,\alpha},
\end{equation*}
for some $C_2>0$, independent on $\alpha$ and $\e$. 

Therefore, we can conclude the proof exactly as done for Theorem \ref{thm:main}.
\end{proof}

Let us stress once again that the only modification with respect to the Allen--Cahn equation with Perona--Malik type diffusion is the fractional time scale, 
which introduces the factor $1/\alpha$ in the exponent.

The same approach applies to the degenerate case $\theta>2$. 
In the classical setting, the lower bound for the energy is algebraic rather than exponentially small, 
and the corresponding slow motion takes place on an algebraically long time scale. 
Using the fractional dissipation inequality \eqref{eq:energyestimate-Q} in place of the classical one, 
the same argument gives the fractional counterpart, with the exponent of the algebraic time scale divided by $\alpha$. 
This leads to the following result.

\begin{thm}[algebraic slow motion for Perona--Malik type diffusion]\label{thm:alg-Q}
Assume that $Q\in C^1(\R)$ satisfies \eqref{eq:Q-ass1}-\eqref{eq:Q-ass2} and that $F\in C^1(\R)$ satisfies \eqref{ipoF1}-\eqref{ipoF2} with $\theta>2$. 
Let $v$ be as in \eqref{vstruct} and let $\left\{k_j\right\}_{j\in\mathbb N}$ be as in \eqref{eq:exp_alg}. 
If $u^\e$ is a sufficiently regular solution of \eqref{eq:Q-dif}-\eqref{eq:Neu}-\eqref{eq:initial} with initial datum satisfying \eqref{eq:ass-u0} and
\begin{equation*}
	\mathscr{E}_{\e}[u^\varepsilon]\leq Nc_\e+C\e^{k_j},
\end{equation*}
for any $\varepsilon\ll1$, then,
\begin{equation*}
	\sup_{0\leq t\leq l{\e^{-k_j/\alpha}}}\|u^\varepsilon(\cdot,t)-v\|_{{}_{L^1}}\xrightarrow[\varepsilon\rightarrow0]{}0,
\end{equation*}
 for any $l>0$.
\end{thm}

The results contained in this last section show that the effect of the Caputo derivative is essentially independent of the particular diffusion operator considered. 
Once the corresponding slow motion result is available for the classical problem, the fractional time derivative modifies the associated time scale by introducing the factor $1/\alpha$. 
In particular, exponential and algebraic slow motions in the classical setting remain exponential and algebraic, respectively, 
while their characteristic time scales are enlarged by the memory effect induced by the fractional derivative.

\section*{Acknowledgements}
The work of D. Cusseddu and R. Folino was partially supported by DGAPA-UNAM, program PAPIIT, grant IN-103425. L. F. L\'opez R\'ios was partially supported by DGAPA-UNAM, program PAPIIT, grant IN113225.

\end{document}